\documentclass[12pt,leqno]{amsart}
\usepackage{a4,amssymb,amsthm,amscd,amsmath, graphicx, verbatim,url,enumerate, mathdots, mathrsfs, cleveref, mathabx}
\usepackage[pdftex,usenames,dvipsnames]{color}
\usepackage{tikz}
\usepackage{tikz-cd}
\usepackage{xypic}

\usepackage[utf8]{inputenc}
\usepackage[T1]{fontenc} 

\title{Strong linkage for function fields of surfaces}

\author{Karim Johannes Becher}
\author{Parul Gupta}
\address{Universiteit Antwerpen, Departement Wiskunde, Middelheim\-laan~1, 2020 Antwerpen, Belgium.}
\email{karimjohannes.becher@uantwerpen.be}
\email{parul.gupta@uantwerpen.be}

\address{Technische Universit\"at Dresden, Institut f\"ur Algebra, 01062 Dresden, Germany.}
\address{IISER Pune, Dr.~Homi Bhabha Road, Pashan, Pune 411 008, India}
\email{parul.gupta@iiserpune.ac.in}
\thanks{This work was supported by the \emph{Fonds Wetenschappelijk Onderzoek -- Vlaanderen} in the \emph{FWO Odysseus Programme} (project \emph{Explicit Methods in Quadratic Form Theory}) and by the \emph{Bijzonder Onderzoeksfonds, University of Antwerp} (project BOF-DOCPRO-4, 2865)}

\date{12.02.2021}

\newcommand{\qq}{\mathbb Q}

\newcommand{\cc}{\mathbb C}

\newcommand{\nat}{\mathbb{N}} 
\newcommand{\zz}{\mathbb Z}

\newcommand{\mc}[1]{\mathcal{#1}}
\newcommand{\mf}[1]{\mathfrak{#1}}

\newcommand{\mg}[1]{{#1}^{\times}}

\newcommand{\brm}[1]{{\sf Br}_m(#1)}
\newcommand{\brmn}[1]{{\sf Br}_m^{\mathsf{nr}}(#1)}
\newcommand{\br}[1]{{\sf Br}(#1)}
\newcommand{\ovl}{\overline}

\newcommand{\mfp}{\mathfrak{p}}
\newcommand{\mfm}{\mathfrak{m}}

\newcommand{\matr}[1]{\mathbb{M}_{#1}}
\newcommand{\s}{\sigma}

\renewcommand{\dim}{\mathsf{dim}}

\renewcommand{\ker}{\mathsf{ker}}
\newcommand{\gal}{\mathsf{Gal}}

\newcommand{\car}{\mathsf{char}}

\newcommand{\supp}{\mathsf{Supp}}

\DeclareMathOperator{\spec}{\mathsf{Spec}}
\newcommand{\N}{\mathsf{N}}
\newcommand{\Hom}{\mathsf{Hom}}

\newcommand{\Div}{\mathsf{Div}}

\newcommand{\mult}{\mathsf{mult}}

\renewcommand{\bmod}{\,\,\mathsf{mod}\,\,}
\renewcommand{\setminus}{\smallsetminus}
\newcommand{\K}{\mathsf{K}}

\swapnumbers
\numberwithin{equation}{section}
\newtheorem{thm}[equation]{Theorem}
\newtheorem{prop}[equation]{Proposition}
\newtheorem{cor}[equation]{Corollary}

\newtheorem{lem}[equation]{Lemma}
\newtheorem{qu}[equation]{Question}
\theoremstyle{definition}

\newtheorem{ex}[equation]{Example}
\newtheorem{exs}[equation]{Examples}
\newtheorem{rem}[equation]{Remark}

\renewenvironment{proof}{\par\noindent {\em Proof:}}{\hfill$\Box$\medskip}
\theoremstyle{plain}

\begin{document}
\begin{abstract}
Over a global field any finite number of central simple algebras of exponent dividing $m$ is split by a common cyclic field extension of degree $m$. We show that the same property holds for function fields of $2$-dimensional excellent schemes over a henselian local domain of dimension one or two with algebraically closed residue field. 

\medskip\noindent
{\sc{Classification (MSC 2010):} 13J15, 16K20, 16S35, 19C30, 19D45}

\end{abstract}

\maketitle

\section{Introduction}

Let $F$ be a field and $m$ a positive integer.
We say that $F$ is \emph{strongly linked in degree $m$} if any finite number of central simple $F$-algebras of exponent dividing $m$  is split by a cyclic field extension of degree dividing $m$. 
It follows from class field theory that global fields have this property.
In \cite{Lenstra}, H.~Lenstra showed a similar statement for elements of the second $K$-group of a global field.

In this article we show that $F$ satisfies strong linkage in degree $m$ in the following three cases:
\begin{enumerate}
\item $F$ is the fraction field of a $2$-dimensional excellent henselian local domain with algebraically closed residue field  of characteristic not dividing~$m$ (e.g.~a finite extension of $\cc(\!(X,Y)\!)$).
\item $F$ is the function field of a curve over the fraction field of an excellent henselian discrete valuation ring with algebraically closed residue field of characteristic not dividing $m$ (e.g.~a finite extension of $\cc(\!(X)\!)(Y)$).
\item $F$ is the function field of a curve over a perfect pseudo-algebraically closed field of characteristic not dividing $m$.
\end{enumerate}

Let $\brm F$ denote the $m$-torsion part of the Brauer group of $F$.
Just as for global fields, in the cases $(1)-(3)$ the period-index problem has a positive answer, that is, any element of $\brm F$ is given by a central simple $F$-algebra of degree~$m$. 
This was shown in \cite{FS89} and  \cite{Ford96} for the cases $(1)$ and $(2)$, respectively, and for the case $(3)$ this follows from \cite[Theorem~3.4]{Efrat}.
The method in all these cases is to study the ramification of elements of $\brm{F}$ with respect to discrete valuations on $F$ and to construct a cyclic field extension that splits the ramification. 
This method relies on the fact that the unramified part of the Brauer group is trivial in these cases. 
We apply the same technique to show strong linkage in degree $m$. 

If $F$ is strongly linked in degree $m$ then it follows that any central $F$-division algebra of exponent $m$ is cyclic of index $m$.
In particular, strong linkage in all degrees implies a positive answer to the period-index problem. 
The field of iterated Laurent series $\cc(\!(X)\!)(\!(Y)\!)(\!(Z)\!)$ gives an example of a field where strong linkage fails in all degrees $m\geq 2$ while the period-index problem has a positive answer.
This strange example is disposed off by restricting to fields of cohomological dimension~$2$, which still covers global fields as well as the fields in the cases~$(1)-(3)$.

It is  more difficult to show that there exist also fields of cohomological dimension~$2$ where the period-index problem has a positive answer but where strong linkage fails in some degree $m$.
A.~Chapman and J.-P.~Tignol showed very recently that $F=\cc(X,Y)$, the rational function field in two variables over the complex numbers, does not satisfy strong linkage for $m=2$.
More specifically, it is shown in \cite{ChTi19} that no quadratic field extension of $F$ splits all the four $F$-quaternion algebras $$(X,Y)_F,\, (X,Y+1)_F,\,(X+1,Y)_F,\,(X,XY+1)_F\,.$$
In view of this result it seems reasonable to expect that strong linkage fails in any degree $m\geq 2$ over this field, and more generally over the function field of any algebraic surface over $\cc$.

The result of \cite{ChTi19} is interesting in view of the fact that the field $\cc(X,Y)$ has the $\mc{C}_2$-property (in terms of Tsen-Lang theory, see \cite[Chapter~5]{Pfister}) and that the period-index problem has a positive answer for this field, by \cite{jong}. 

The structure of this article is as follows.
In \Cref{S:Milnor}, we will recall the setup of Milnor $K$-theory and formulate the problem of strong linkage in these terms. 
This seems the most natural context for this problem and for some of the tools that we need, and it further gives an opening to studying the analogous problem for higher $K$-groups. 
In this article, however, we focus on the second $K$-group of a field and its quotient modulo $m$, which in characteristic not dividing $m$ is naturally isomorphic to $\brm F$, by the Merkurjev-Suslin Theorem. 
We recall this relation in \Cref{S:cyclic}, where we further collect the relevant results on the vanishing of the unramified part of the Brauer group of certain function fields.
In \Cref{S:Brauer-unramified}, a general strategy is provided for finding a common slot for a finite set of symbols in $K$-theory modulo $m$ (or of the corresponding cyclic algebras).
In \Cref{S:PAC}, we treat case $(3)$.
For treating the cases $(1)$ and $(2)$ we need some results from algebraic geometry. 
In these cases we need to achieve that the ramification of the set of symbols is contained in a normal crossing divisor on some regular model of the field $F$.
Sections \ref{S:ram-2dim} and \ref{S:ram-surfaces} revisit some results related to divisors on regular surfaces over henselian local domains.
In \Cref{S:main-result} we show strong linkage in the cases $(1)$ and $(2)$.
In this situation one can obtain a candidate for this slot as a function passing along all components of the support of the ramification divisor with multiplicity $1$.
Finally, in \Cref{S:disc-qf} we make the observation that a discrete valuation on a quasi-finite field is henselian and the residue field is algebraically closed of characteristic zero. 
As a consequence we obtain that algebraic function fields over a discretely valued quasi-finite field fall under case $(2)$, and therefore they satisfy strong linkage in all degrees. 

The techniques that we use to obtain our results are not new. They have been developed and applied to show structure properties for central simple algebras and quadratic forms over function fields of certain surfaces, by Ford and Saltman (\cite{FS89}, \cite{Ford96}, \cite{Saltman}, \cite{Sal07}) and further by Colliot-Th\'el\`ene, Ojanguren, Parimala  and Suresh (\cite{PS98}, \cite{CTOP}, \cite{PS10}). Here, we are focussing on function fields of cohomological dimension $2$ and hence do not cover the case of function fields of $p$-adic curves, which were in the focus of several of the articles just mentioned.
We shift the focus from the study of a single object like a central simple algebra, an element of a Milnor $K$-group or of a Galois cohomology group to the problem of having a simultaneous representation of an arbitrary finite number of such objects. It is probably clear to experts that the methods carry over to settle that problem, leading to most of our results. However, to achieve a good presentation, we include an exposition of these tools, and we further adapt them to the set-up of Milnor $K$-theory as far as possible. 
For the crucial fact that splitting ramification splits the elements in question (which also holds in some other situations than those considered here) we rely entirely on the existing vanishing results for the unramified part of certain cohomology groups.
On the other hand, the computations on ramification become straightforward in the setup of generators and relations, and we underline this by including these computations.
\smallskip

This article is based on Parul Gupta's PhD-thesis prepared under the supervision of Karim Johannes Becher (\emph{Universiteit Antwerpen}) and Arno Fehm (\emph{Technische Universit\"at Dresden}) in the framework of a joint PhD at \emph{Universiteit Antwerpen} and \emph{Universit\"at Konstanz}.

\section{$K$-groups and ramification}\label{ramification} 

\label{S:Milnor}
We recall some basic terminology and facts from Milnor $K$-theory.
Our main references are Milnor's seminal article \cite{Mil70} and \cite[Chapter~7]{GS}.
\medskip

Let $F$ be a field. Let $\mg{F}$ denote the multiplicative group of $F$. 
For $n,m \in \nat $, let $\K_n^{(m)}F$ be the \emph{$n$th Milnor $K$-group of $F$ modulo $m$}, that is, the additive abelian group defined by generators and relations as follows:
$\K_n^{(m)}F$ is generated by so-called \emph{symbols} $\{a_1,\ldots, a_n\}$ with parameters $a_1, \ldots, a_n\in \mg F$, and the defining relations are 
that the map $(\mg F)^n \rightarrow \K_n^{(m)}F,(a_1, \ldots, a_n) \mapsto \{a_1, \ldots, a_n\}$ is multilinear,
that $\{a_1, \ldots, a_n\} =0 $ holds for any $a_1,\dots,a_n\in\mg F$ with $a_i+a_{i+1} =1$ for some index $i<n$, 
and that $m\cdot \K_n^{(m)}F=0$.
Note that $\K_n^{(0)}F$ is the full Milnor $K$-group, usually denoted by $\K_n F$, and that $\K_n^{(m)}F=\K_nF/m\K_nF$.  
For a field extension $F'/F$ there is a natural homomorphism $\K_n^{(m)}F\to \K_n^{(m)}F'$ which maps the symbol $\{a_1,\dots,a_n\}$ 
given by parameters $a_1,\dots,a_n\in\mg{F}$ to the symbol with the same notation in $\K_n^{(m)}F'$; this map is generally not injective, because symbols from $\K_n^{(m)}F$ are subject to additional relations in $\K_n^{(m)}F'$.
Given a field extension $F'/F$ and $\alpha\in \K_n^{(m)}F$, we denote by $\alpha_{F'}$ the image of $\alpha$ under the natural homomorphism $\K_n^{(m)}F\to \K_n^{(m)}F'$.

We now fix a positive integer $m$.
The direct sum of $\zz$-modules $\bigoplus_{n\in\nat}\K_n^{(m)}F$ 
has a natural multiplication which makes it into a graded $\zz$-algebra;
the multiplication is determined by the rule that, for $n,n'\in\nat$ and $a_1,\dots,a_n,b_1,\dots,b_{n'}\in\mg F$, we have
$$\{a_1,\dots,a_n\}\cdot \{b_1,\dots,b_{n'}\}=\{a_1,\dots,a_n,b_1,\dots,b_{n'}\}\,.$$

Let $\alpha \in \K_n^{(m)}F$. 
A symbol $\s\in \K_r^{(m)}F$ is called a \emph{factor of $\alpha$} if $\alpha=\s\cdot \gamma$ for some element $\gamma\in \K_{n-r}^{(m)}F$.
For the case where $r=1$,
we also say that $a\in\mg{F}$ is a \emph{slot of $\alpha$} if the symbol $\{a\}$ is a factor of $\alpha$.

Given a symbol $\s\in \K_r^{(m)}F$, 
we denote by $\s\cdot \K_{n-r}^{(m)}F$ the subgroup of $\K_n^{(m)}F$ consisting of the elements having $\s$ as a factor, that is, the 
image of the map $\K_{n-r}^{(m)}F\to \K_n^{(m)}F$ given by multiplication from the left by $\s$. 
We say that $\K_n^{(m)}F$ is \emph{strongly linked} if for every finite subset $\mc S$ of $\K_n^{(m)}F$ there exists a
symbol $\s\in \K_{n-1}^{(m)}F$ such that $\mc{S}\subseteq\s \cdot \K_1^{(m)}F$.
Trivially,  if  $\K_n^{(m)}F=0$ then $\K_n^{(m)}F$ is strongly linked. 
It is also obvious that, if $\K_n^{(m)}F$ is strongly linked, then every element of $\K_n^{(m)}F$ is a symbol.

In this article we focus on the case $n=2$, hence on the problem of strong linkage for $\K_2^{(m)}F$.
When $F$ is a finite field, then $\K_2^{(0)}F=0$ (see \cite[Example 7.1.3]{GS}).
In \cite{Lenstra}, H.~Lenstra showed that $\K_2^{(0)}F$ is strongly linked in the case where $F$ is a global field. 
Apart from these fields and their algebraic extensions, 
 we do not expect strong linkage to hold  for $\K_2^{(0)}F$ in any other cases, and we therefore consider the problem of strong linkage of $\K_2^{(m)}F$ in the sequel for $m>0$.

A crucial tool in the context of Milnor $K$-theory is given by residue maps with respect to discrete valuations.
Our main reference for valuation theory is \cite{EP05}.
By a \emph{$\zz$-valuation} we mean a valuation with value group $\zz$. 

Let $v$ be a $\zz$-valua\-tion on $F$. We denote by $\mathcal O_v, \mf m_v$ and $ 
\kappa_v$ the corresponding valuation ring, its maximal ideal and its residue field, respectively.
Given $a\in \mc O_v$ we write $\bar a$ for the residue $a+\mfm_v$ in $\kappa_v$.
We recall the definition of the ramification homomorphism with respect to $v$, which is denoted by $\partial_v$.

\begin{prop}\label{Milnorresiduehom} 
For each $n\geq 1$ there exists a unique homomorphism
$$\partial_v: \K_n^{(m)}F\rightarrow \K_{n-1}^{(m)}\kappa_v$$ such that, for all $x \in \mg F$ and $u_2, \ldots u_n \in  \mg{\mc O}_v$, one has
$$ \partial_v(\{x, u_2, \ldots u_n\}) = v(x)\cdot \{\bar u_2, \ldots, \bar u_n \} \text{~in~}\K_{n-1}^{(m)}\kappa_v\,.$$
\end{prop}

\begin{proof}
See \cite[Proposition~7.1.4]{GS} for the case $m=0$; this implies the  statement for any $m\in\nat$.
\end{proof}

For $n=2$ the homomorphism $\partial_v :\K_2^{(m)}F \rightarrow \K_1^{(m)}\kappa_v$ from \Cref{Milnorresiduehom}  is given on symbols by the  rule
$$\qquad\{a,b\} \mapsto(-1)^{v(a)v(b)}\{ \overline {a^{-v(b)} b^{v(a)}}\}~~~~ \mbox{for $a,b \in F^\times$}\,.$$

Recall that the valuation $v$ on $F$ is \emph{henselian} if it extends uniquely to every finite field extension of $F$, or equivalently, if for any polynomial $f\in\mc{O}_v[X]$ any simple root of $f$ in $\kappa_v$ is the residue of a root of $f$ in $\mc{O}_v$ (see \cite[Theorem 4.1.3]{EP05}).

\begin{prop}\label{cdvfk2} 
Assume that $v$ is henselian, $\car(\kappa_v)$ does not divide $ m$ and $\K_{n}^{(m)}{\kappa_v} = 0$. Then $\partial_v:\K_{n}^{(m)} {F}\to \K_{n-1}^{(m)}{\kappa_v}$ is an isomorphism. Furthermore, any element $a\in\mg{F}$ such that $v(a)$ coprime to $m$ is a slot of every element of $\K_{n}^{(m)}{F}$.
\end{prop}

\begin{proof} 
By \cite[Corollary 7.1.10]{GS}, we obtain that $\partial_v:\K_{n}^{(m)} {F}\to \K_{n-1}^{(m)}{\kappa_v}$ is an isomorphism: in fact, while the statement of \cite[Corollary 7.1.10]{GS} assumes $v$ to be complete, the proof only uses  that $v$ is henselian.  
Consider $a\in\mg{F}$ with $v(a)$ coprime to $m$. 
Then $\partial_v(\{a\}\cdot\K_{n-1}^{(m)}F)=v(a)\cdot \K_{n-1}^{(m)}\kappa_v=\K_{n-1}^{(m)}\kappa_v$, and as $\partial_v$ is injective, we conclude that $a$ is a slot of every element of $\K_n^{(m)}F$.
\end{proof}

Let $m$ and $n$ be positive integers with $m \geq 2$ and
let $\alpha\in \K_n^{(m)}F$. 
We call $\partial_v(\alpha)$ the \emph{ramification of $\alpha$ at $v$}. 
We call \emph{$\alpha$ unramified at $v$} if $\partial_v(\alpha) = 0$ in $\K_{n-1}^{(m)}\kappa_v$, otherwise we say that $\alpha$ is \emph{ramified at $v$}.
We say that a finite field extension $L/F$ \emph{splits the ramification of $\alpha$ at $v$} if for every $\zz$-valuation $w$ on $L$ with $\mc O_v \subseteq \mc O_w$ we have $\partial_w(\alpha_L) =0$.

\begin{lem}\label{L:splitram}
Let $f\in\mg{F}$ be such that $v(f)$ is coprime to $m$ and let $L/F$ be a finite field extension $L/F$ with $f\in L^{\times m}$.
Then $L/F$ splits the ramification at $v$ of all elements of $\K_n^{(m)}F$.
\end{lem}
\begin{proof}
Consider a $\zz$-valuation $w$ of $L$ such that $\mc{O}_w\cap F=\mc{O}_v$.
We have that $[w(\mg{L}):w(\mg{F})]\leq [L:F]<\infty$.
We set $e=[w(\mg{L}):w(\mg{F})]$.
Since $v(f)$ is coprime to $m$ and $f\in L^{\times m}$, it follows that $m$ divides $e$.
By \cite[Remarks 7.1.6 (2)]{GS}, we conclude that $\partial_w(\alpha_L)=e\cdot \partial_v(\alpha)=0$ in $\K_n^{(m)}L$ for every $\alpha\in\K_n^{(m)}F$.
\end{proof}

\section{Symbol algebras and linkage}\label{S:cyclic} 

Let $F$ be a field and $m$ a positive integer.
In this section, we connect strong linkage of $\K_2^{(m)}F$ to strong linkage of $F$ in degree $m$ as we defined it in the introduction, namely in terms of central simple algebras. This is done via the Merkurjev-Suslin Theorem, which relates the Brauer group to the second Milnor $K$-group of a field. 
Our main reference for the Brauer group of a field is~\cite{GS}.
\medskip

The Brauer group of $F$ and its $m$-torsion part are denoted by $\br{F}$ and by $\brm{F}$, respectively. 
Let $A$ be a central simple $F$-algebra.
We denote  by $[A]$ its  class in $\br{F}$.
Recall that $[A]=0$ if and only if $A$ is \emph{split}, i.e.~ isomorphic to  $\matr{d}(F)$ for some positive integer $d$. 
For a field extension $L/F$, any central simple $F$-algebra $A$ gives rise to a central simple $L$-algebra $A_L=A\otimes_FL$, and this  
gives rise to a natural homomorphism $\br{F} \rightarrow \br{L}$ defined by $[A] \mapsto [A_L]$, which restricts to a homomorphism  $\brm{F} \rightarrow \brm{L}$.

We assume in the sequel that $\car(F)$ does not divide $m$ and we denote by $\omega$ a primitive $m$th root of unity
contained in an algebraic extension of $F$.
For $a\in\mg F$,
the splitting field $X^m-a$ over $F$ contains $\omega$, and we denote this extension by $F(\omega,\sqrt[m]{a})$, or simply by 
$F(\sqrt[m]{a})$ when $\omega\in F$.
Recall that, if $\omega\in F$, then every cyclic field extension of $F$ of degree $m$ is of the form $F(\sqrt[m]{a})$ for some $a\in\mg{F}$. 

If $\omega\in F$, then for $a, b\in \mg F$ we denote by $(a,b)_{F,\omega}$ the 
$F$-algebra generated by two elements $x,y$ subject to the relations 
$$x^m=a,\, y^m=b \mbox{ and } yx = \omega xy\,.$$ 
Such algebras are called \emph{symbol algebras}.
By \cite[Chapter 11, Theorem 1]{Dra}, for $a,b\in\mg F$ the $F$-algebra $(a,b)_{F,\omega}$ is central simple of degree $m$.

%
%
%
%
%
%
%
%
%
%
%
%
%


\begin{thm}[Merkurjev-Suslin]\label{MS}
If $\omega\in F$, then
the rule $\{a,b\} \mapsto [(a,b)_{F,\omega}]$ for $a,b\in\mg F$ determines 
an isomorphism $$\Phi_\omega: \K_2^{(m)}F \rightarrow \brm{F}\,.$$
\end{thm}

\begin{proof}
See \cite[Theorem 2.5.7]{GS}.
\end{proof}

We need the following slightly extended version of \cite[Corollary~4.7.6]{GS}.
For a finite field extension $F'/F$ we denote by $\N_{F'/F}:F'\to F$ the norm map. 

\begin{lem}\label{firstslotalg}
Let $A$ be a central simple $F$-algebra of degree $m$ and $a\in \mg F$. 
Assume that $\omega\in F$ and that $A_{F(\!\sqrt[m]{a})}$ is split.
Then $A \simeq (a,b)_{F, \omega}$ for some  $b \in \mg F$.
\end{lem}

\begin{proof}
Set $z=\sqrt[m]{a}$, $L=F(z)$ and $m'=[L:F]$. 
Then $z^m=a$ and $m'$ divides $m$. 
Note that $\N_{L/F}(z)=\omega^r \cdot z^{m'}$ for some $r\in\nat$ and hence $z^{m'} \in \mg F$. 
We set $a'=z^{m'}$ and $s=\frac{m}{m'}$ to obtain that $a'^s=a$ and $L=F(\!\sqrt[m']{a'})$.
As $[L:F]=m'$, \cite[Corollary~4.7.6]{GS} yields that $[A]=[(a',b)_{F,\omega^s}]$ in $\br{F}$ for some $b\in \mg F$, and using 
\cite[Chapter 11, Lemma 5 and 6]{Dra} we conclude that $[A]=[(a,b)_{F,\omega}]$.
\end{proof}

\begin{thm}\label{firstslotalgKgroupgen}
Assume that $[F(\omega) :F]$ is coprime to $m$ and let $a\in\mg{F}$. 
The kernel of the natural homomorphism $\K_2^{(m)} F \rightarrow \K_2^{(m)} {F(\omega,\sqrt[m]{a}})$ consists of the symbols  $\{a,b\}$ with $b\in \mg F$.  
\end{thm}

\begin{proof}
Set $F' = F(\omega)$. Then $F'(\sqrt[m]{a})= F(\omega,\sqrt[m]{a})$.
Clearly, for any $b\in\mg{F}$, the symbol $\{ a, b \}$ belongs to $\ker (\K_2^{(m)} F \rightarrow \K_2^{(m)} F'(\sqrt[m]{a}))$. 
For the converse, consider $\alpha \in \ker (\K_2^{(m)} F \rightarrow \K_2^{(m)} F'(\sqrt[m]{a}))$.
Since $\Phi_\omega(\alpha_{F'(\sqrt[m]{a})})= 0$, we obtain
by \Cref{firstslotalg} that $\Phi_\omega(\alpha_{F'}) = [(a, c)_{F',\omega}]$ for some $c \in \mg{F'}$. Since $\Phi_\omega$ is injective, we have that $\alpha_{F'} = \{a,c\}$ in $\K_2^{(m)}F'$. 
By \cite[Chapter 7, Remark 7.3.1]{GS}, it follows that 
$[F':F]\cdot \alpha = \{a, \N_{F'/F}(c)\}$. 
Since $[F':F]$ is coprime to $m$, there exists $r\in\zz$ with 
$r\cdot[F':F]  \equiv 1\bmod{m\zz}$. Then $\alpha = \{a, \N_{F'/F}(c) ^r\}$ in $\K_2^{(m)}F$.
\end{proof}

\begin{cor}\label{C:K-linkage-cyclic-ext}
Assume that $[F(\omega) :F]$ is coprime to $m$. 
Then $\K_2^{(m)}F$ is strongly linked if and only if any finite subgroup of $\K_2^{(m)}F$ is contained in the kernel of the natural map $\K_2^{(m)}F\to\K_2^{(m)}F(\omega,\sqrt[m]{a})$ for some $a\in\mg{F}$.
\end{cor}
\begin{proof}
This follows directly from the definition and \Cref{firstslotalgKgroupgen}.
\end{proof}

\begin{cor}
Assume that $\omega\in F$.
Then $\K_2^{(m)}F$ is strongly linked if and only if  every finite number of central simple $F$-algebras of degree $m$ is split by a cyclic field extension of $F$ of degree dividing $m$.
\end{cor}

\begin{proof}
This follows 
 from \Cref{firstslotalg}, \Cref{MS} and \Cref{C:K-linkage-cyclic-ext}.
\end{proof}

\section{The unramified Brauer group} 
\label{S:Brauer-unramified}  

In this section we look at the unramified Brauer group of a field and 
recall some statements from the literature. The purpose is to reformulate results on the unramified Brauer group in the way we will use them later. 

Let $F$ be a field and let $\Omega_F$ be the set of all $\zz$-valuations on $F$. Let $m\in \nat$ be a positive integer. 
For $v\in\Omega_F$ we have $v(m)=0$ if and only if $\car(\kappa_v)$ does not divide $m$.
We set $$\Omega_F^{(m)}=\{v\in\Omega_F\mid v(m)=0\}\,.$$

Let $v \in \Omega_F^{(m)}$. There is a well-known ramification homomorphism (see \cite[Chapter 10, page 68]{Saltman}) 
$$\partial_v:  \brm {F} \rightarrow \Hom(\gal(\kappa_v), \zz/m\zz) ,$$
where $\mathsf{Gal}(\kappa_v)$ denotes the absolute Galois group of $\kappa_v$, which is endowed with the profinite topology, and where $\Hom(\gal(\kappa_v), \zz/m\zz)$ is the set of all continuous homomorphisms $\gal(\kappa_v) \rightarrow \zz/m\zz$. 
Let $\hat{F}_v$ denote the completion of $F$ with respect to $v$ and let $\hat v$ denote the unique  unramified extension of $v$ to $\hat{F}_v$. 
The map $\partial_v$ factors as
$$\brm {F} \rightarrow \brm {\hat{F}_v}\stackrel{\partial_{\hat v}}\rightarrow \Hom(\gal(\kappa_v), \zz/m\zz),$$
where the first map is the scalar extension homomorphism.
In particular,
$$ \ker ( \brm {F} \rightarrow \brm {\hat{F}_v}) \subseteq \ker(\partial_v)\,.$$

For the definition of $\br{\mc O_v}$, the Azumaya-Brauer group of the discrete valuation ring $\mc O_v$, we refer to \cite[Chapter 3]{Saltman}. Here we only need this group to link two important statements in the literature.
The following statement characterizes the $m$-torsion part $\brm{\mc O_v}$ of $\br{\mc O_v}$ as the kernel of the ramification map~$\partial_v$.

\begin{thm}[Auslander-Brumer]\label{brauerresiduemap}
For $v\in\Omega_F^{(m)}$ the following natural sequence is exact:
$$0\rightarrow \brm {\mc O_v} \rightarrow \brm F \stackrel{\partial_v} \rightarrow \Hom(\gal(\kappa_v), \zz/m\zz) \rightarrow 0  $$
\end{thm}

\begin{proof}
The statement follows from \cite[page 289]{AB}.
\end{proof}

We say that \emph{$\alpha \in \brm F$ is unramified at $v$} if $\partial_v(\alpha) =0$.
Consider now an arbitrary subset $\Omega \subseteq \Omega_F^{(m)}$. 
We set
$$\brmn{F, \Omega} = \{\alpha \in \brm F\mid \partial_v(\alpha)=0 \mbox{ for all } v\in \Omega\}.$$ 
We further abbreviate $$\brmn{F} =\brmn{F, \Omega_F^{(m)}}.$$
Obviously, we have that  $\brmn{F}\subseteq \brmn{F, \Omega}$ for any $\Omega\subseteq \Omega_F^{(m)}$. 

\begin{prop}\label{completionramficationeasylemma}
Let $\Omega \subseteq \Omega_F^{(m)}$ be such that $\brm{\kappa_v} =0$ for all $v\in\Omega$.
Then
$$ \brmn{F, \Omega} = \ker\left ( \brm {F} \rightarrow \prod_{v\in \Omega} \brm {\hat{F}_v}\right).$$
\end{prop}

\begin{proof}
Consider $v \in \Omega$.
By \Cref{brauerresiduemap} and \cite[Corollary 8.5]{AG} we have that
$\ker (\partial_{\hat v}) = \brm {{\mc O}_{\hat v}}\simeq  \brm{\kappa_v}=0$.
Hence $\ker(\partial_v) = \ker ( \brm {F} \rightarrow  \brm {\hat{F}_v})$. 
Using this for all $v\in\Omega$, we obtain the statement.
\end{proof}

\begin{prop}\label{LGPK-group}
Assume that $F$ contains a primitive $m$th root of unity $\omega$.
Let $\Omega\subseteq \Omega_F^{(m)}$ be such that $\brmn{F,\Omega}=0$. Then the homomorphism 
\begin{eqnarray*}
\partial_\Omega: & \K_2^{(m)}F \rightarrow \bigoplus_{v\in \Omega} \K_1^{(m)}\kappa_v, & \alpha \mapsto (\partial_v(\alpha))_{v\in \Omega}
\end{eqnarray*}
is injective.
\end{prop}

\begin{proof}
Let $\alpha\in\ker(\partial_\Omega)$.
By the injectivity of the map $\Phi_\omega:\K_2^{(m)}F \rightarrow \brm{F}$ from \Cref{MS}, in order to show that $\alpha=0$ it suffices to verify that $\Phi_\omega(\alpha)=0$, and since $\brmn{F,\Omega}=0$ it is enough to check that $\partial_v(\Phi_\omega(\alpha))=0$ for all $v\in\Omega$.

Consider $v\in\Omega$ and fix $\pi\in\mg F$ with $v(\pi)=1$.
Then
$\alpha=\sum_{i=1}^r\{a_i,b_i\}+\{\pi,c\}$ with $r\in\nat$ and $a_1,b_1,\dots,a_r,b_r,c\in\mg{\mc{O}_v}$.
Hence $\{\ovl{c}\}=\partial_v(\alpha)=0$ in $\K_1^{(m)}\kappa_v$, whereby $\ovl{c}\in\kappa_v^{\times m}$.
As $v\in\Omega\subseteq \Omega^{(m)}$, $\car(\kappa_v)$ does not divide $m$.
Hence $c\in \hat{F}_v^{\times m}$, by Hensel's Lemma.
Thus $\{\pi, c\} =0$ in $\K_2^{(m)}\hat{F}_v$ and $\Phi_\omega(\alpha_{\hat{F}_v}) = \bigotimes_{i=1}^r[ (a_i,b_i)_{\hat{F}_v,\omega}]\in\brm{\mc{O}_{\hat v}}=\ker(\partial_{\hat v})$. Since $\partial_v$ factors over $\partial_{\hat v}$, we obtain that $\partial_v(\Phi_\omega(\alpha))=0$.
\end{proof}

\section{Function fields of curves with trivial Brauer group} 
\label{S:PAC} 

Let $E$ be a field. 
By an \emph{algebraic function field} over $E$ we mean a finitely generated field extension $F/E$ of transcendence degree $1$.
In this section we outline a strategy to show strong linkage for $\K_2^{(m)}F$ for an algebraic function field $F/E$ and for $m\in\nat$ not divisble by $\car(E)$.
This strategy will be applicable when the base field $E$ satisfies a very strong property, which is given for example when $E$ is a finite field.
\medskip

For a field extension $F/E$ we set $$\Omega_{F/E}=\{v\in\Omega_F\mid v(\mg{E})=0\}\,.$$
We fix a positive integer $m$ which is not a multiple of $\car(E)$ and denote by $\omega$ a primitive $m$th root of unity contained in an algebraic extension of $E$.

\begin{lem}\label{HMtoSL}
Assume that $[E(\omega):E]$ is coprime to $m$.
Let $F/E$ be an algebraic function field. 
Assume that $\brmn{F, \Omega_{F/E}} =0$ and $\brmn{F', \Omega_{F'/E}} =0$ for every finite field extension $F'/F(\omega)$ of degree $m$.
Then the following hold:
\begin{enumerate}[$(a)$]
\item If $\alpha \in \K_2^{(m)}F$ and $f \in \mg F$ is such that $v(f)$ is coprime to $m$ for every $v \in \Omega_{F/E}$ with $\partial_v(\alpha) \neq 0$, then $f$ is a slot of $\alpha$. 
\item  $\K_2^{(m)}F$ is strongly linked.
\end{enumerate}
\end{lem}

\begin{proof}
Set $\Omega = \Omega_{F/E}$.

$(a)$  Let $f\in\mg{F}$ and set $S=\{v\in\Omega\mid v(f)\mbox{ coprime to }m\}$.
Let $\alpha \in \K_2^{(m)}F$ be such that $\partial_v(\alpha)=0$ for all $v\in\Omega\setminus S$.
If $S=\emptyset$ then $\alpha\in\brmn{F,\Omega}=0$, and thus $\alpha=\{f,1\}$.
Suppose now that $S\neq \emptyset$ and set $F' = F(\omega,\sqrt[m]{f})$.
Note that $\omega\in F'$.
We claim that $[F':F(\omega)]=m$.
We choose any $v\in S$ and use that $v(f)$ is coprime to $m$ and $v$ is unramified in $F(\omega)$ to conclude that $[F':F(\omega)]=m$.
The hypothesis therefore yields that $\brmn{F' , \Omega_{F'/E}} =0$. 
In order to show that $f$ is a slot of $\alpha$, 
it suffices by \Cref{firstslotalgKgroupgen} to show that $\alpha_{F'} =0$, and since $\brmn{F', \Omega_{F'/E}} =0$ it is further sufficient by \Cref{LGPK-group} to show 
that $\alpha_{F'}$ is unramified with respect to $\Omega'=\Omega_{F'/E}$. 

Consider $w\in \Omega_{F'/E}$.
Then $\mc{O}_w\cap F=\mc{O}_v$ for a unique $v\in \Omega_{F/E}$.
If $\partial_v(\alpha)=0$, then clearly $\partial_w(\alpha_{F'})=0$.
If $\partial_v(\alpha)\neq 0$, 
then $v \in S$, whereby $v(f)$ is coprime to $m$, and as $f\in F'^{\times m}$,  it follows by \Cref{L:splitram} that $\partial_w(\alpha_{F'}) =0$.

$(b)$ Consider an arbitrary finite subset $\mc S \subseteq \K_2^{(m)}F$.
Note that, for any $g\in\mg F$, there are only finitely many valuations $v\in\Omega$ such that $v(g)\neq 0$.
Hence, the subset
$S = \{v\in\Omega\mid \partial_v(\alpha)\neq 0\mbox{ for some }\alpha\in \mc{S}\}$
of $\Omega$ is finite.
By the Weak Approximation Theorem  \cite[Theorem 2.4.1]{EP05}, there exists $f \in \mg F$ such that $v(f) =1$ for all $v\in S$.
Then $f$ is a slot of every element of $\mc{S}$, by $(a)$.
This shows that $\K_2^{(m)}F$ is strongly linked. 
\end{proof}

\begin{ex}\label{GFSL}
Let $E$ be a finite field which contains a primitive $m$th root of unity.
Let $F/E$ be an algebraic function field.
It follows by the Albert-Brauer-Hasse-Noether Theorem (see \cite[Theorem 32.11]{Reiner}) and by \Cref{completionramficationeasylemma} that $\brmn{F} =0$.
Hence it follows by \Cref{HMtoSL} that $\K_2^{(m)}F$ is strongly linked.
\end{ex}

We now apply the same method to show strong linkage for algebraic function fields over PAC-fields. 
Recall that a field $E$ is called \emph{pseudo algebraically closed} or a \emph{PAC-field} if every absolutely irreducible algebraic variety over $E$ has an $E$-rational point.
For a general discussion of these fields we refer to \cite[Chapter 11]{FJ}.
Note that, if $E$ is a PAC-field, then by \cite[Corollary 11.5.5]{FJ} the value group of any valuation on $E$ is divisible, so in particular
$\Omega_E=\emptyset$.

\begin{thm}[Efrat]\label{Ef} 
Let $E$ be a PAC-field such that $\car(E)$ does not divide~$m$. 
Then $\brmn{F} =0$ for any algebraic function field $F/E$.
\end{thm}

\begin{proof}
Let $F/E$ be an algebraic function field.
By \cite[Corollary 3.2]{Efrat} and \cite[Remark 3.5(a)]{Efrat}, the natural homomorphism $\brm F \rightarrow \oplus_{v\in \Omega_F} \brm {\hat{F}_v}$ is injective. 
Since $E$ is a PAC-field, we have $\Omega_E=\emptyset$ and
therefore $\Omega_F=\Omega_{F/E}$.
For $v\in\Omega_F$, 
then the residue field $\kappa_v$ is a finite extension of $E$ and hence itself a PAC-field, whereby $\brm{\kappa_v}=0$.
Furthermore, the hypothesis on $E$ yields that 
$\Omega_F=\Omega_F^{(m)}$.
Hence we obtain by \Cref{completionramficationeasylemma} that 
$\brmn{F} =0$. 
\end{proof}

\begin{thm}\label{PACSL} 
Let $E$ be a PAC-field such that $\car(E)$ does not divide~$m$. 
Assume that $[E(\omega):E]$ is coprime to $m$ for a primitive $m$th root of unity $\omega$.
Then $\K_2^{(m)}F$ is strongly linked for every algebraic function field $F/E$.
\end{thm}
\begin{proof}
The hypotheses on $E$ imply that $\Omega_E=\emptyset$ and
 $\Omega_{F'}=\Omega_{F'/E}=\Omega_{F'}^{(m)}$ for every algebraic function field $F'/E$. 
Hence, by \Cref{Ef} we have that
 $\brmn{F',\Omega_{F'/E}}=\brmn{F'} = 0$  for every algebraic function field $F'/E$.
 Therefore the statement follows by \Cref{HMtoSL}.
\end{proof}

\section{Ramification on $2$-dimensional regular local rings} 
\label{S:ram-2dim} 

Let $R$ be a $2$-dimensional regular local ring. 
In particular, $R$ is a unique factorisation domain (see~e.g.~\cite[Theorem~20.3]{Ma-CRT}), the maximal ideal of $R$ is generated by two elements and it is the only non-principal prime ideal of $R$. A pair of generators of the maximal ideal is a called a \emph{parameter system of $R$}. Any element occurring in a parameter system of $R$ is called a \emph{parameter}, and it is necessarily a prime element of $R$.

We denote by $\mc P_R$ the set of all height-one prime ideals of $R$.
Let $\mf p \in \mc P_R$.
The localisation $R_\mfp$ is a discrete valuation ring and its residue field $\kappa_{\mfp} = R_{\mf p}/{\mfp} R_{\mfp}$ is the fraction field of $R/{\mfp}$, which is a  $1$-dimensional noetherian local domain.
For $x \in R\setminus {\mf p}$ let $l_R(R/({\mf p}+ Rx))$ denote the length of the $R$-module $R/({\mf p}+ Rx)$. By \cite[Lemma 7.1.26]{Liu}, the rule $(x+{\mf p}) \mapsto l_R(R/({\mf p}+ Rx))$ for $x \in R\setminus {\mf p}$ defines a group homomorphism $\mult_{\mf p} : \mg{\kappa}_{\mf p} \rightarrow \zz$.

Fix now a positive integer $m$ which is invertible in $R$.
For $\mf p \in \mc P_R$, composition of $\mult_{\mf p}$ with the residue map $\zz\to\zz/m\zz$ determines a homomorphism
$$r_{\mf p}: \K_1^{(m)} \kappa_{\mf p} \rightarrow \zz/m\zz,\,\,\{x\}\mapsto \mult_{\mf p}(x)+m\zz\,.$$
Compiling these maps for all $\mfp\in\mc P_R$, we obtain a group homomorphism
$$ r_R : \bigoplus_{{\mf p}\in \mc P_R}\K_1^{(m)}\kappa_{\mf p}\rightarrow \zz/m\zz,\,\,\, (a_{\mf p})_{{\mf p}\in \mc P_R} \mapsto \sum_{{\mf p}\in \mc P_R} r_{\mf p}(a_{\mf p})\,.$$

Let $F$ denote the fraction field of $R$.
For $\mfp \in \mc P_R$, the $\zz$-valuation $v_\mfp$ on $F$ with $\mc{O}_{v_\mfp}=R_\mfp$ induces a ramification homomorphism 
$\partial_{\mf p}=\partial_{v_\mfp}: \K_2^{(m)}F \rightarrow \K_{1}^{(m)}\kappa_\mfp$.
Compilation of the maps $\partial_\mfp$ for all $\mfp\in  \mc P_R$ yields a homomorphism
$$\partial_R= \bigoplus_{\mfp\in \mc P_R} \partial_\mfp: \K_2^{(m)}F \rightarrow \bigoplus_{\mfp\in \mc P_R}\K_{1}^{(m)}\kappa_\mfp\,.$$

A crucial tool is the following well-known reciprocity law,
 formulated in \cite[Lemma 1.1]{Sal97} in terms of Galois cohomology.
We include a proof, working in $\K_2^{(m)}F$.

\begin{prop}\label{reciprocity2} 
Let $R$ be a $2$-dimensional regular local ring with $m\in\mg{R}$.
The homomorphism $r_R\circ \partial_R: \K_2^{(m)}{F} \rightarrow \zz/m\zz$ is trivial. 
\end{prop}

\begin{proof} 
We set $\partial=\partial_R$ and $r=r_R$. As $R$ is a unique factorization domain,  $\K_2^{(m)} {F}$ is generated by symbols of the following three types: $\{a,b\},\{\pi, b\}$ and $\{\pi, \delta\}$, where $a,b \in \mg R$ and $\pi, \delta$ are non-associated prime elements of $R$. Thus it suffices to show that $(r\circ \partial) (\alpha) =0$, where $\alpha$ is a symbol of one of these three types.

\underline{Case 1:} If $\alpha =\{a,b\}$ with $a,b \in \mg R$, then $\partial(\alpha)=0$  and hence $(r\circ \partial)(\alpha) =0$. 

\underline{Case 2:} Let $\alpha =\{\pi ,b\}$ with $b \in \mg R$ and a prime element $\pi$ of $R$. Then $\partial _{\mf p}(\alpha) =0$ in $\K_1^{(m)}\kappa_{\mf p}$ for all $\mf p \in \mc P_R \setminus \{R\pi\}$ and $\partial _{R\pi}(\alpha) = \{\ovl{b}\}$ in $\K_1^{(m)}\kappa_{R\pi}$. As $\pi R +bR =R$, we have $l_R(R/(R\pi  +Rb))=0$ and thus $r_{R\pi}(\{\ovl{b}\}) =0$. Hence $(r \circ \partial)( \alpha) =0$. 
 
\underline{Case 3:} Let $\alpha =\{\pi ,\delta\}$ with two non-associated prime elements $\pi$ and $\delta$ of $R$. Then $\partial _{\mf p}(\alpha) =0$ for all $\mf p \in \mc P_R\setminus \{R\pi , R\delta \}$. We further have 
\begin{align*}
(r \circ \partial _{R\pi} )(\alpha) &  =  (r_{R\pi} \circ \partial _{R\pi} )(\alpha)   =r_{R\pi} (\bar{\delta}) = l_R(R/(R\pi+R\delta))\,\mbox{ and }\\
(r \circ \partial _{R\delta} ) (\alpha) & =  (r_{R\delta} \circ \partial _{R\delta})(\alpha) =  r_{R\delta}(\bar{\pi}^{-1}) =-l_R(R/(R\pi+R\delta))\,.
\end{align*}
Thus $(r\circ \partial )(\alpha) =0$ in this case, too.
\end{proof}

\begin{cor}[Saltman]\label{ramification-one-parameter}
Let $R$ be a $2$-dimensional regular local ring with $m\!\in\!\mg{R}$\!\!. 
Let $\alpha \in \K_2^{(m)}{F}$, let $\pi $ be a parameter of $R$ and $\mfp=R\pi$. 
Assume that $\partial_{\mf q}(\alpha) =0 $ for all $\mf q \in \mc P_R\setminus \{ \mfp\}$.
Then $\partial_{\mfp}( \alpha) = \{\bar u\} \mbox { in }  \K_1^{(m)} \kappa_{\mfp}$ for some $u \in \mg R$. 
\end{cor}
 
\begin{proof}
 Let $u \in R\setminus \mfp$ be such that $\partial_{\mfp}(\alpha) =\{\bar u\}$ in  $\K_1^{(m)} \kappa_{\mfp}$ and let $\delta \in R$ be such that $(\pi, \delta) $ is a parameter system of $R$. 
Then $\delta+\mfp$ is generates the maximal ideal of the discrete valuation ring $R/\mfp$, so $u \equiv \delta ^l b\bmod \mfp$ for some $b\in \mg R$ and $l \in \nat$.   

By the hypothesis and \Cref{reciprocity2}, we obtain that 
$(r_{\mfp} \circ \partial_{\mfp})(\alpha) =0\mbox{ in }\zz/m\zz\,.$ Hence  
$\ovl{l}=r_{\mfp}(\{\bar {\delta}^l \bar{b}\}) = r_{\mfp}(\{\bar u\}) =  0\mbox{ in }\zz/m\zz\,,$
so $m$ divides $l$.
Hence $\{\ovl{u}\}=\{\bar b\}$ in $\K_1^{(m)}\kappa_{\mfp}$, which shows that we may replace $u$ by $b$ to achieve that $u\in\mg{R}$.
\end{proof}

Given a field $F$ and a subring $R$ of $F$, we 
set $$\Omega_{F/R}=\{v\in\Omega_F\mid R\subseteq \mc{O}_v\}\,.$$
(This extends the notation $\Omega_{F/E}$ for a field extension $F/E$ given in \Cref{S:PAC}.)

We will need the following consequence of a purity result on the Brauer group due to Auslander-Goldman, which is translated here to $K$-groups.

\begin{lem}\label{AGpurity-cor2}
Let $R$ be a $2$-dimensional regular local ring with $m\in\mg{R}$ and let $F$ be its fraction field. Assume that $F$ contains a primitive $m$th root of unity $\omega$.
Let $\alpha\in \mathsf{K}_2^{(m)}F$ be such that $\partial_\mfp(\alpha)=0$ for all $\mfp\in\mc{P}_R$.
Then $\partial_v(\alpha)=0$ holds for all $v\in \Omega_{F/R}$.
\end{lem}

\begin{proof}
Let $\Omega=\{v_\mfp\mid \mfp\in\mc{P}_R\}$ and note that $\Omega\subseteq \Omega_F^{(m)}$.
By  \cite[Proposition 7.4]{AG}, we have that 
$$\br R=\bigcap_{v\in \Omega} \br{ \mc O_v}\,.$$ 
Hence, for the $m$-torsion part we obtain by \Cref{brauerresiduemap} that
$$\brm{R}=\brmn{F, \Omega}\,.$$

Let $\xi=\Phi_\omega(\alpha)$ where $\Phi_\omega $ is the homomorphism in \Cref{MS}. 
It follows by the hypothesis on $\alpha$  
that $\xi\in\brmn{F, \Omega}=\brm{R}$.
Let $v\in\Omega_{F/R}$.
Since $R\subseteq\mc{O}_v$ it follows that $\brm{R}\subseteq \brm{\mc{O}_v}$.
Hence $\xi\in\brm{\mc{O}_v}$ and therefore $\partial_v(\xi)=0$.
This means that $\partial_v(\alpha)=0$.
 \end{proof}

We need also  the following well-known consequence of the previous statements.

\begin{lem}\label{curvepointram}
Let $R$ be a $2$-dimensional regular local ring with fraction field $F$, maximal ideal $\mfm$ and residue field $\kappa$. 
Assume that $m\in\mg{R}$ and $\kappa^\times = \kappa^{\times m}$. 
Assume that $F$ contains a primitive $m$th root of unity. 
Let $\alpha \in \K_2^{(m)}{F}$. 
Assume that there exists a parameter $\pi$ of $R$ such that $\partial_{\mf q}(\alpha) =0 $ for all $\mf q \in \mc P_R\setminus \{ R\pi\}$. 
Then $\partial_v(\alpha)=0$ for all $v\in\Omega_{F/R}$ with $\mfm_v\cap R=\mfm$. 
\end{lem}

\begin{proof}
Let $\mfp=R\pi$.
Since $\partial_{\mf q}(\alpha) =0 $ for all $\mf q \in \mc P_R\setminus \{\mfp\}$ it follows by \Cref{ramification-one-parameter} that $\partial_{\mfp}(\alpha) =\{\bar u\}$ in  $\K_1^{(m)} \kappa_{\mfp}$ for some $u \in \mg R$. 
Set 
$\beta = \alpha -\{ \pi , u \}$.
Then $\partial_R(\beta) =0$.  By \Cref{AGpurity-cor2}, we have that $\partial_v(\beta) =0$ for every $v \in {\Omega}_{F/R}$.

Consider now a valuation $v\in\Omega_{F/R}$ with $\mfm_v\cap R=\mfm$. 
Then $\kappa\subseteq \kappa_v $.
Since $u\in\mg{R}\subseteq\mg{\mc{O}}_v$ and $\bar u\in \mg{\kappa}=  \kappa^{\times m}  \subseteq {\kappa}^{\times m}_v$, we obtain that
$$\partial_v(\{\pi, u\})  = v(\pi)\{ \bar u \} =0 \mbox{ in $\K_1^{(m)}\kappa_v$} \,,$$ 
whereby
$\partial_v(\alpha) = \partial_v(\beta) + \partial_v(\{\pi, u\}) =0$ in $\K_1^{(m)}\kappa_v$.
\end{proof}

\section{Ramification on surfaces} \label{S:ram-surfaces} 

Our main reference for results and terminology from algebraic geometry is \cite{Liu}.
By a \emph{surface} we mean an integral separated noetherian scheme of dimension~$2$.

Let $\mathscr X$ be a surface and let $F$ be the function field of $ \mathscr X$.
For a point $x$ on $\mathscr{X}$ let $\mc O_{\mathscr{X},x}$ denote the stalk of the structure sheaf $\mc{O}_\mathscr{X}$ at $x$, which is a local domain with fraction field $F$. 
A point $x$ of $\mathscr X$ is \emph{regular} (resp.~\emph{normal}) if the local ring $\mc O_{\mathscr X,x}$ is regular (resp.~integrally closed). 
We say that $\mathscr X$ is \emph{regular} (resp.~\emph{normal}) if $\mathscr X$ is regular (resp.~normal) at every point of $\mathscr X$.
For $i \in \{0,1,2\}$ we denote by $\mathscr X^i$ the set of all points of codimension $i$ of $\mathscr X$.
Note that, if $\mathscr{X}=\spec(R)$ for a $2$-dimensional regular local ring $R$, then $\mathscr{X}^1=\mc{P}_R$.

A \emph{prime divisor on $\mathscr X$} is an irreducible closed subset of $\mathscr X$ of codimension~one, 
and is equal to the closure $\ovl{\{x\}}$ for a unique point $x\in \mathscr X^1$.
Let $\Div (\mathscr{X})$ be the free abelian group generated by the prime divisors of $\mathscr{X}$. 
The elements of $\Div (\mathscr{X})$ are called \emph{divisors on $\mathscr{X}$}.
Let $D \in \Div (\mathscr X)$.  For $x\in {\mathscr X}^1$ we denote by $n_x(D)$ the coefficient of the prime divisor $\ovl{\{x\}}$ in $D$.
We set 
$$ \supp_\mathscr{X} (D) = \{x \in \mathscr{X}^1 \mid  n_x(D) \neq 0\}\,.$$
A prime divisor $D'$ is called a \emph{component of $D$} if $\supp_\mathscr{X} (D') \subseteq \supp_\mathscr{X} (D)$.  
For $x\in \mathscr{X}$, we say that \emph{$D$ contains $x$} or that \emph{$x$ lies on $D$} if there is a component of $D$ that contains~$x$. 
Note that, for $x\in \mathscr X^1$, $\ovl{\{x\}}$ is the only prime divisor on $\mathscr X$ that contains $x$, and hence $x$ lies on $D$ if and only if $n_x(D)\neq 0$.

Assume now that $\mathscr X$ is normal. 
For every $x\in \mathscr X^1$, since $\mc O_{\mathscr X,x}$ is $1$-dimensional and integrally closed, hence it is a discrete valuation ring of $F$, and we denote the corresponding $\zz$-valuation on $F$ by $v_x$ and its residue field by $\kappa_x$.
For $f\in \mg F$ the set $\{x\in\mathscr{X}^1\mid v_x (f) \neq0\}$ is finite (see \cite[Lemma~7.2.5]{Liu}), and we set
$$(f)_{\mathscr X} = \sum_{x\in\mathscr{X}^1} v_x(f) \cdot \ovl{\{x\}} $$
and call this the \emph{principal divisor given by $f$}.

\begin{lem}\label{L:new}
Let $D\in\Div({\mathscr X})$.
Let $S$ be a finite set of regular points of $\mathscr X$ contained in an open affine subscheme of $\mathscr X$.
Then there exists $h\in\mg{F}$ such that no point from $S$ lies on the divisor
$D-(h)_{\mathscr X}$.
\end{lem}
\begin{proof}
Let $U$ be an open affine subscheme of $\mathscr X$ containing $S$.
Set $A = \mc O_{\mathscr X}(U)$. 
Let $M$ be the set of the prime ideals of $A$ corresponding to points in $S$. We set $S = A\setminus (\bigcup_{\mfp \in M}\mfp )$. Then $A_S$ is a regular semi-local ring and $F$ is the fraction field  of $A_S$.
By a well-known extension of the Auslander-Buchsbaum Theorem, every regular semi-local ring is a unique factorization domain; see e.g.~\cite[Lemma 0.12]{Sal07}.
Hence every divisor on $\spec(A_S)$ is principal. 
Hence the image of $D$ under the natural map $\Div(\mathscr X) \rightarrow \Div(\spec(A_S))$ is attained by a principal divisor $(h)_\mathscr{X}$, for some $h\in\mg{F}$. 
The divisor $D'=D-(h)_{\mathscr X}$ on ${\mathscr X}$ maps to zero in $\Div(\spec(A_S))$.
For any $x\in \mathscr{X}^1$, the prime divisor $\ovl{\{x\}}$ maps to zero in $\Div(\spec(A_S))$ if and
only if it contains no point from $S$.
Thus no point contained in $S$ lies on $D'$.
\end{proof}

Let $x$ be a regular point of $\mathscr{X}$ and let $D\in \Div(\mathscr X)$.
As a special case of \Cref{L:new}, we obtain that there exists some $f\in\mg{F}$ such that $D-(f)_{\mathscr X}$ does not contain $x$, and we call such an element $f$ a \emph{local equation of  $D$ at $x$}.
 
We call a point $x\in \mathscr{X}^2 $ a \emph{crossing point of $D$} if $x$ lies on more than one component of $D$. 
Let $x \in \mathscr X^2$ be a regular point of $\mathscr X$. 
We say that $D$ has \emph{normal crossing at $x$} if a local equation  of $D$ at $x$ can be given in the form $\pi^i\delta^j$ with a parameter system $(\pi, \delta)$ of $\mc O_{\mathscr{X},x}$ and some $i,j \in \nat$. 
We say that \emph{$D$ has normal crossings on $\mathscr X$} if $D$ has normal crossing at every point $x\in\mathscr{X}^2$.

Let $R$ be a noetherian domain of dimension at most $2$. 
By a surface over $R$ we mean an $R$-scheme $\mathscr X$ which is a surface.

Let $R$ be a henselian local noetherian domain with fraction field $F$. 
Let $\eta: \mathscr{X} \rightarrow \spec (R)$ be a surface over $R$.
We denote by $\mathscr{X}_x$ the $\kappa_x$-scheme $\mathscr{X} \times_{\spec(R)}\spec (\kappa_x)$, called the \emph{fiber of $\eta$ over $x$}.
By \cite[Proposition 3.1.16]{Liu}, the underlying topological space of $\mathscr{X}_x$ is naturally homeomorphic to $\eta^{-1}(x)$. 
In the sequel, we denote by $s$ the closed point of $\spec (R)$. Then $\mathscr{X}_s$ is called the \emph{special fiber of~$\eta$}.

\begin{prop} \label{basicproperties}
Let $R$ be a henselian local noetherian domain of dimension one or two. Let $\mathscr{X}$ be a surface over $R$ whose structure morphism $\eta: \mathscr{X} \rightarrow \spec(R)$ is proper and surjective. Then all closed points of $\mathscr{X}$ lie on $\mathscr{X}_s$. Furthermore, $\dim(\mathscr{X}_s)\leq 1$ and $\mathscr{X}_s$ is connected. In particular, $\mathscr{X}_s$ is either a closed point of $\mathscr X$ or a finite union of $1$-dimensional irreducible subsets of $\mathscr{X}$. 
\end{prop}

\begin{proof} 
Since $\eta$ is proper, for any closed point $x$ of $\mathscr{X}$, $\eta(x)$ is closed in $\spec(R)$, whereby $\eta(x) =s$. 
Thus $\mathscr{X}_s$ contains all closed points of $\mathscr{X}$.

By \cite[Lemma 2.3.17]{Liu}, $\spec (\kappa_s) \rightarrow \spec (R)$ is a closed immersion.
It follows by \cite[Proposition 3.2.4~$(a)$ and $(c)$]{Liu} that $\mathscr{X}_s \rightarrow \mathscr{X}$ is of finite type. Since $\mathscr{X}$ is noetherian, so is $\mathscr{X}_s$, by \cite[Exercise 3.2.1]{Liu}.
Hence $\mathscr{X}_s$ has finitely many irreducible components.
As $\eta$ is surjective, we have $\eta^{-1}(s) \neq \mathscr X$. Since $\mathscr{X}$ is irreducible and $\mathscr{X}_s$ is closed in $\mathscr{X}$,  we get by \cite[Proposition 2.5.5]{Liu} that $\dim \mathscr{X}_s \leq 1 < \dim \mathscr{X}$. As $R$ is a henselian local ring and $\eta$ is proper, ${\mathscr{X}}_s$ is connected, by \cite[p.~135, Proposition 18.5.19]{EGA}. If $\dim{\mathscr{X}}_s =0$, then ${\mathscr{X}}_s$ is a closed point on $\mathscr{X}$. If $\dim {\mathscr{X}}_s =1$, then it follows that all irreducible components of ${\mathscr{X}}_s$ are of dimension $1$.
\end{proof}

In the situation of \Cref{basicproperties}, if  $\dim({\mathscr{X}}_s) =1$, then we write
$\supp_{\mathscr{X}}(\mathscr{X}_s)$ for 
 the finite set of points $x \in \mathscr X^1$ for which the closure $\ovl{\{x\}}$ is a component of~$\mathscr X_s$.

\begin{thm}[Colliot-Th\'el\`ene-Ojanguren-Parimala]\label{CTOP}
Let $R$ be a henselian local domain of dimension one or two with $m\in\mg{R}$.
Assume that the residue field of $R$ is either finite or separably closed.
Let $\mathscr X$ be a regular projective surface over $R$ such that $\mathscr X \rightarrow \spec(R)$ is surjective and let $F$ be the function field of $\mathscr X$. 
Then $\brmn{F, \Omega_{\mathscr X}} =0$, where $\Omega_{\mathscr X} =\{ v_x \mid x \in \mathscr X^1\}$. 
 \end{thm}

\begin{proof}
By \Cref{basicproperties}, the special fiber of $\mathscr X$ has dimension at most one.
Then $\brmn{F, \Omega_{\mathscr X}} = \brm{\mathscr X}$ (See \cite[Theorem 6.1~$(b)$]{G2} or \cite[Theorem 1.2]{KC}).
Since $\mathscr X$ is regular, we obtain by \cite[Corollary 1.10 and Corollary 1.11]{CTOP} that $\brm{\mathscr X}=\br{\mathscr X} =0$.
\end{proof}

Let $\mathscr{X}$ be a surface. An $\mathscr X$-scheme $\mathscr X'$ is called a \emph{model of $\mathscr X$} if the structure morphism $\mathscr X' \rightarrow \mathscr X$ is birational and proper.

We refer to \cite[Definition 8.2.35]{Liu} for the definition of \emph{excellence} for a ring as well as for a scheme, and also for the fact that for any excellent noetherian commutative ring $R$ the scheme $\spec(R)$ is excellent.
Given an excellent surface the existence of a regular model is ensured by a result due to Lipman. 

\begin{thm}[Lipman]\label{desingularisation} 
Let $\mathscr{X}$ be an excellent surface. 
There exists a regular projective model $\eta:\mathscr{X}'\to \mathscr X$.
Moreover, given an effective divisor $D$ one can choose $\eta:\mathscr{X}'\to \mathscr X$ such that $\eta^{-1}(D)$ has only normal crossings on $\mathscr{X}'$.
\end{thm}

\begin{proof}
See \cite[Theorem]{Lip78} for  the existence of a regular model $\eta:\mathscr{X}'\to \mathscr X$ and  the explanation in \cite[p.~193]{Lip75} for the fact that it can be chosen such that  $\eta^{-1}(D)$ has only normal crossings on $\mathscr{X}'$, for a given divisor $D$ on $\mathscr{X}$. 
Using further \cite[Corollaire 5.6.2]{GroEGA}, one can achieve that $\eta$ is projective.
\end{proof}

The following examples cover the situations where we will apply \Cref{desingularisation}.

\begin{exs}\label{excellentrem}
By \cite[Theorem 8.2.39, Corollary 8.2.40]{Liu}, every complete noetherian local ring as well as every Dedekind domain of characteristic zero is excellent.
Furthermore, by \cite[Theorem 8.2.39 $(c)$]{Liu}, given an excellent noetherian ring $R$, any projective  surface over $\spec(R)$ is excellent.
\end{exs}

\begin{cor}\label{lemmamadetorefer} 
Let $E$ be a field and let $v$ be a discrete valuation on $E$. 
If $\car(E)\neq 0$ then assume that $v$ is complete.
Let $F$ be an algebraic function field over $E$.
Then $F$ is the function field of a regular projective surface $\mathscr X$ over $\mc O_v$ such that $\mathscr X\rightarrow \spec(\mc O_v)$ is surjective.
\end{cor}
\begin{proof}
By \cite[Proposition 7.3.13]{Liu}, there exists a unique normal projective curve $C$ over $E$ with function field $F$. 
Using \cite[Example 10.1.4]{Liu} and the first steps of \cite[Proposition 10.1.8]{Liu}, we obtain a projective flat surface $\mathscr Y $ over $\mc O_v$ with function field $F$. In particular $\mathscr Y\rightarrow \spec(\mc O_v)$ is surjective. 
Note that $\mc{O}_v$ is a local Dedekind ring and either of characteristic zero or complete.
Since $\mathscr Y \rightarrow \spec(\mc O_v)$ is projective, we obtain by \Cref{excellentrem} that $\mathscr Y$ is excellent. 
By \Cref{desingularisation}, there exists a regular projective model $\mathscr X\to \mathscr Y$. 
As  $\mathscr X \rightarrow \mathscr Y$ then is projective and birational, it follows that $\mathscr X\rightarrow \spec(\mc O_v)$ is projective and surjective and that $F$ is the function field of $\mathscr X$. 
\end{proof}

\section{Function fields of surfaces}\label{S:main-result} 

In this section we consider function fields of surfaces over a complete local domain of dimension one or two with algebraically closed residue field. 
Typical examples are the fields $\cc(\!(t)\!)(X)$ and $\cc(\!(X,Y)\!)$ and their finite extensions.
It was shown in \cite{Ford96} and \cite{FS89} that any central simple algebra over such a field is cyclic.
In \cite[Theorem 2.1]{CTOP}, using techniques developed in \cite{Sal97}, a simpler proof of the main result of \cite[Theorem 1.6]{FS89} was given.

Here we will show that these fields are strongly linked in any degree coprime to the residue characteristic. 
We mainly follow the method of \cite[Theorem 2.1]{CTOP}. 

\medskip

Consider a normal surface $\mathscr X$ and let $F$ be the function field of $\mathscr X$.
Let $m$ be a positive integer.
Any $x\in\mathscr{X}^1$ gives rise to a $\zz$-valuation $v_x$ on $F$ and therefore to a ramification homomorphism 
$\partial_x= \partial_{v_x}: \K_2^{(m)}F \rightarrow \K_{1}^{(m)}\kappa_x$.
For $\alpha \in \K_2^{(m)}F$, we set $\supp_{\mathscr{X}}(\alpha) = \{ x \in \mathscr X^{1}\mid \partial_x(\alpha) \neq 0\}$.
We set $$\Omega_{F/\mathscr X} = \{ v \in \Omega_F \mid  {\mc O}_{\mathscr X,x}\subseteq \mc{O}_v \mbox{ for some }x\in {\mathscr X}\} \,.$$

\begin{rem}\label{R:val-cen-X}
Consider $v\in\Omega_{F/\mathscr X}$. Let $x\in\mathscr{X}$ be such that $\mc{O}_{\mathscr{X},x}\subseteq\mc{O}_v$.
Then $\mfm_v\cap\mc{O}_{\mathscr{X},x}$ is a nonzero prime ideal of $\mc{O}_{\mathscr{X},x}$, hence either a height-one prime ideal or equal to $\mfm_{\mathscr{X},x}$. In particular, if $\mfm_v\cap\mc{O}_{\mathscr{X},x}\neq\mfm_{\mathscr{X},x}$, then 
$x\in\mathscr{X}^2$ and $\mc{O}_v$ is the localisation of $\mc{O}_{\mathscr{X},x}$ at $\mfm_v\cap\mc{O}_{\mathscr{X},x}$ and therefore equal to $\mc{O}_{\mathscr{X},x'}$ for some $x'\in\mathscr{X}^1$. Hence, in any case we may choose $x\in\mathscr{X}$ with $\mc{O}_{\mathscr{X},x}\subseteq\mc{O}_v$ in such way that $\mfm_v\cap\mc{O}_{\mathscr{X},x}=\mfm_{\mathscr{X},x}$.
\end{rem}

The following lemma is distilled from the proof of \cite[Theorem 2.1]{CTOP}.

\begin{lem}\label{unramifiedsymbols}
Let $\mathscr X$ be an excellent regular surface and let $F$ be its function field.
Assume that $F$ contains a primitive $m$th root of unity. 
Assume that, for every closed point $x \in \mathscr X$, $\car (\kappa_x)$ does not divide $m$ and 
$\kappa_x^\times = \kappa_x^{\times m}$. 
Let $f \in \mg F$ and let $D_1,D_2\in \Div(\mathscr X)$ be such that $ (f)_{\mathscr X} = D_1 +D_2$ and 
$$\supp_{\mathscr X} (D_1) \cap \supp_{\mathscr X} (D_2) = \emptyset \,.$$ 
Assume that $D_1$ has normal crossings on $\mathscr X$, $n_x(D_1)=1$ for every $x\in \supp_{\mathscr X}(D_1)$ and $D_2$ does not contain any crossing point of $D_1$.
Let $\alpha \in \K_2^{(m)}F$ be such that $\supp_{\mathscr X}(\alpha) \subseteq \supp_{\mathscr X}(D_1)$. 
Then $F(\!\sqrt[m] {f})/F$ splits the ramification of $\alpha$ at every $v \in \Omega_{F/\mathscr X}$.
\end{lem}

\begin{proof}
Let $F'=F(\!\sqrt[m]{f})$.
Let $v\in \Omega_{F/\mathscr X}$.
Consider an arbitrary $\zz$-valuation $w$ on $F'$ with 
$\mc{O}_w\cap F=\mc{O}_v$.
We need to show that $\partial_w(\alpha_{F'})=0$ in $\K_1^{(m)}\kappa_w$.

If $\partial_v(\alpha)=0$ in $\K_1^{(m)}\kappa_v$, then this is obvious. 
We assume now that $\partial_v(\alpha)\neq 0$. 
By \Cref{R:val-cen-X} we may choose $x \in \mathscr X$ such that $\mc O_{\mathscr X,x} \subseteq \mc{O}_v$ and $\mf m_v \cap  \mc O_{\mathscr X,x}= \mf m_{\mathscr X,x}$.

Assume first that $x\in \mathscr X^1$. Then $v=v_x$ and $x\in \supp_{\mathscr{X}}(\alpha)\subseteq\supp_{\mathscr{X}}(D_1)$.
We obtain that $v(f)=v_x(f)=n_x(D_1)=1$. It follows by \Cref{L:splitram}  that $\partial_w(\alpha_{F'}) =0$.

Assume now that $x\in \mathscr X^2$. Set $R = \mc O_{\mathscr X,x}$, whereby $v \in \Omega_{F/R}$.
As $\partial_v(\alpha)\neq 0$, we have $\partial_R(\alpha)\neq 0$, by \Cref{AGpurity-cor2}.
Hence $R$ has a prime element $\pi$ such that $\partial_{R\pi}(\alpha)\neq 0$.
Since $D_1$ has normal crossings on $\mathscr X$ and $\supp_{\mathscr{X}}(\alpha)\subseteq\supp(D_1)$, we obtain that $\pi$ is a parameter of $R$.
Set $\mfp=R\pi$.
Since $\partial_v(\alpha)\neq 0$, it follows by \Cref{curvepointram} that $\partial_{\mf q}(\alpha) \neq 0$ for some $\mf q \in \mc P_R\setminus\{ \mfp \}$. 
In particular, $x$ is a crossing point of $D_1$.
By the hypotheses on $D_1$ and $D_2$, it follows that $x$ does not lie on $D_2$. 
We set $\delta=\pi^{-1}f$.
Since the divisor $D_1$ has normal crossings and all its components have multiplicity $1$, we obtain that $(\pi, \delta)$ is a parameter system of $R$ and $\mf q=R\delta$.

Fix $\lambda , \mu \in R\setminus\{0\}$ such that  $\partial_{\mfp}(\alpha) =\{ \bar \lambda \}$ in $\K_1^{(m)}\kappa_{\mfp}$ and $\partial_{\mf{q}}(\alpha) = \{\bar \mu\}$ in $\K_1^{(m)}\kappa_{\mf q}$. 
Then $R/\mfp$  is a discrete valuation ring whose maximal ideal is generated by $\delta+\mfp$.
Hence there exists $\nu\in\mg{R}$ and $s\in \nat$ such that $\lambda\equiv \nu\delta^s \bmod \mfp$.
Exchanging the roles of $\pi$ and $\delta$, the same argument yields that $\mu \equiv \nu' \pi^r \bmod {\mf q}$ 
for some $\nu'\in\mg{R}$ and $r\in\nat$.
Using \Cref{reciprocity2}, 
we obtain that $r+s \equiv 0 \bmod m\zz$.  
Set  
$$ \beta = \alpha - \{\pi, \nu \}- \{\delta, \nu' \}- s \{ \pi, \delta\}\,.$$
We obtain that $\partial_R(\beta) =0$. 
It follows by \Cref{AGpurity-cor2} that  $\partial_v(\beta)=0$.
Note that the symbols $\{\pi,\nu\}$ and $\{\pi,\pi\}$ lie in $\ker(\partial_{\mfp'} )$ for any $\mf p' \in \mc P_R\setminus\{\mfp \}$. 
Similarly, $\{ \delta , \nu' \}$ lies in $\ker(\partial_{\mfp'})$ for any $\mfp'\in \mc P_R\setminus\{ \mf q \}$. 
We obtain by \Cref{curvepointram} that 
$$\partial_v(\{\pi,\nu\}) =\partial_v(\{\pi,\pi\})=\partial_v(\{ \delta , \nu' \}) =0  \mbox{ in } \K_1^{(m)}\kappa_v\,.$$ 
We conclude that $$\partial_w(\beta_{F'} ) = \partial_w (\{\pi,\nu\}) = \partial_w(\{\pi,\pi\})=\partial_w(\{ \delta , \nu' \}) =0  \mbox{ in } \K_1^{(m)}\kappa_w\,.$$
Since $f = \pi\delta$, we have
$$\{ \pi, \delta\} = \{\pi ,f \} - \{\pi , \pi\}\,,$$
and as $\{\pi ,f \} = 0$ in $\K_2^{(m)}F'$, we conclude that $\partial_w(\{\pi, \delta \}) =0$ in $\K_1^{(m)}\kappa_w$.
This shows that $ \partial_w(\alpha_{F'}) = 0$ in $\K_1^{(m)}\kappa_w$.
\end{proof}

\begin{thm}\label{maintheorem}
Let $R$ be an excellent henselian local domain of dimension  one or two with separably closed residue field of charactersitic not dividing~$m$.
Let $\mathscr X$ be a projective surface over $R$ such that $\mathscr X \rightarrow \spec(R)$ is surjective. 
Let $F$ be the function field of $\mathscr X$.
Then $\K_2^{(m)}F$ is strongly linked.
\end{thm}

\begin{proof} 
The hypothesis implies that $R$ contains a primitive $m$th root of unity.

Let $\mc S \subseteq \K_2^{(m)}F$ be a finite subset. 
To show that $\K_2^{(m)}F$ is strongly linked, we need to show that all elements $\mc{S}$ have a common slot.

Since $\mathscr X \rightarrow \spec(R)$ is projective and since $R$ is excellent, 
it follows from \cite[Theorem 8.2.39 (c)]{Liu} that $\mathscr X$ is excellent as well.
We set 
$$\supp_\mathscr{X} ({\mc S}) = \bigcup_{\alpha\in {\mc S}} \supp_\mathscr{X} (\alpha)\,.$$ 
By \Cref{basicproperties}, $\supp_\mathscr{X}(\mathscr{X}_s)$ is finite. 
Hence $Z = \supp_\mathscr{X}(\mc S) \cup \supp_{\mathscr{X}}(\mathscr{X}_s) $ is a finite set.
By \Cref{desingularisation},  there exists a regular projective model $\eta:\mathscr X'\to\mathscr X$ such that $\eta^{-1}\left(\sum_{z\in Z}\ovl{\{z\}}\right)$ has normal crossings on $\mathscr {X}'$. 

Let $Z' =\supp_\mathscr{X'}(\mc S) \cup \supp_\mathscr{X'}(\mathscr{X}'_s) $ and  
set 
$$D = \sum_{x \in Z'} \ovl{\{x\}}\,.$$ 
By the choice of $\eta:{\mathscr{X}'}\to {\mathscr{X}}$ the divisor $D$ on ${\mathscr{X}'}$ has normal crossings.

We fix  a finite set $\mathscr P$ of closed points of $\mathscr{X}'$ containing all crossing points of $D$ and at least one point on each component of $D$. By the Approximation Theorem \cite[2.4.1]{EP05}, there exists $g \in \mg F$ such that $v_{x}(g) =n_x(D)$ for every $x \in \supp_{\mathscr{X}'}(D)$. Let $D''=(g)_{\mathscr{X}'} - D$.
Since $\mathscr{X}' \rightarrow \spec (R)$ is projective, by \cite[Proposition 3.3.36~$(b)$]{Liu} there exists an affine open subset $U$ of ${\mathscr{X}'}$ containing~$\mathscr P$.
By \Cref{L:new},
 there exists $h\in\mg{F}$ such that no point in $\mathscr P$ is contained in $D''-(h)_{\mathscr{X}'}$.
Set $D'=D''- (h)_{\mathscr{X}'}$ and $f= gh^{-1}\in\mg{F}$. 
Then $D+D'=(f)_{\mathscr{X}'}$ and  $\supp_{\mathscr{X}'}(D') ~\cap~ \supp_{\mathscr{X}'}(D) = \emptyset$.

We claim that $f$ is a slot of every element of $\mc S$.
Consider $\alpha\in\mc S$.
Since $\supp_{\mathscr X'}(\alpha) \subseteq \supp_{\mathscr X'}(D)$,
it follows by \Cref{unramifiedsymbols} that $F(\sqrt[m]{f})$ splits the ramification of $\alpha$ over $\Omega_{F/\mathscr X'}$.
Let $\mathscr Y$ be the normalization of $\mathscr X'$ in $F(\sqrt[m]{f})$.  
Since $F(\sqrt[m]{f})/F$ is a finite separable field extension, 
it follows by \cite[Proposition 4.1.25]{Liu} that $\mathscr Y \rightarrow \mathscr X'$ is a finite morphism, 
in particular it is projective.
Hence $\mathscr Y$ is excellent, by \cite[Theorem 8.2.39 $(c)$]{Liu}. 
By \Cref{desingularisation}, there exists a regular projective model $\mathscr{Y'}\to\mathscr Y$.  
Since restriction of a valuation in $\Omega_{F(\sqrt[m]{f})/\mathscr Y'}$ on $F$ is equivalent to a valuation in $\Omega_{F/\mathscr X}$, it follows by \Cref{unramifiedsymbols} that $\partial_{v_y}(\alpha_{F(\sqrt[m]{f})}) =0$ for all $y \in \mathscr Y'^{1}$ .
Since $\mathscr Y' \rightarrow  \mathscr X'$ and $\mathscr X' \rightarrow \spec(R)$ are projective and surjective, we have that $\mathscr Y'$ is a regular projective surface over $R$ such that $\mathscr Y' \rightarrow \spec(R)$ is surjective.
Thus by \Cref{CTOP}, we have that $\brmn{F(\sqrt[m]{f}), \Omega_{\mathscr Y'}} =0$.
It follows by \Cref{LGPK-group} that $\alpha_{F(\sqrt[m]{f})} =0$. 
We conclude that $f$ is a slot of $\alpha$, by \Cref{firstslotalgKgroupgen}. 
\end{proof}

\begin{cor}\label{arithmeticcase}
Let $E$ be the fraction field of an excellent henselian discrete valuation ring with separably closed residue field of characteristic not dividing $m$. 
Let $F/E$ be an algebraic function field. 
Then $\K_2^{(m)}F$ is strongly linked.
\end{cor}

\begin{proof}
Let $R$ denote the valuation ring.
Since $R$ is excellent, either $R$ is complete or $\car(R)=0$.
By \Cref{lemmamadetorefer}, there exists a regular projective surface $\mathscr X$ over $R$ such that $\mathscr X\rightarrow \spec(R)$ is surjective and such that $F$ is the function field of $\mathscr X$. 
Hence the result follows from \Cref{maintheorem}. 
\end{proof}

\begin{exs}\label{arithmeticcaseex}
Let $m\in \nat $ be a positive integer. Consider the following cases:
\begin{enumerate}[$(1)$] 
\item $E = k(\!(X)\!)$ for an algebraically closed field $k$ of characteristic coprime to $m$.
\item $E$ is the maximal unramified extension of the field $\qq_p$ for a prime number $p$ not dividing $m$.
\end{enumerate}
In each of these cases, $E$ is a complete discretely valued field and the corresponding discrete valuation ring is excellent, by \Cref{excellentrem}.
Hence, it follows by \Cref{arithmeticcase} that $\K_2^{(m)} F$ is strongly linked for any algebraic function field $F/E$.
\end{exs}

\begin{cor}\label{henseliancase}
Let $F$ be the fraction field of a $2$-dimensional excellent hense\-lian local domain $R$ with separably closed residue field and such that $m\in\mg{R}$.
Then $\K_2^{(m)}F$ is strongly linked.  
\end{cor}

\begin{proof}
As $R$ is $2$-dimensional and excellent, \Cref{desingularisation} asserts the existence of a regular projective model of $\spec(R)$, so the result follows from \Cref{maintheorem}. 
\end{proof}

\begin{ex}
Let $m$ be a positive integer. 
Consider $F=k(\!(X,Y)\!)$ where $k$ is an algebraically closed field of characteristic coprime to $m$.
Then $F$ is the fraction field of the $2$-dimensional complete regular local domain 
$k [\![X,Y]\!]$, 
which is excellent by \Cref{excellentrem}. 
Hence $\K_2^{(m)}F$ is strongly linked, by \Cref{henseliancase}.
\end{ex}

\section{Discretely valued quasi-finite fields}\label{S:disc-qf} 

A field is called \emph{quasi-finite} if it has a finite field extension of every degree and all its finite field extensions are cyclic.
Equivalently, a field is quasi-finite if it is perfect and its absolute Galois group is isomorphic to the procyclic group $\hat \zz$. 
Finite fields are quasi-finite.
Another natural example of a quasi-finite field is $\cc(\!(t)\!)$, the field of Laurent series in one variable over $\cc$. 

Let $m$ be a positive integer.
Recall that $\K^{(m)}F$ is strongly linked for any global field $F$, hence in particular when $F$ is an algebraic function field over a finite field.
We further have seen in \Cref{arithmeticcaseex} that the same holds for algebraic function fields over $\cc(\!(t)\!)$.
It is well-known that quasi-finite fields can have different properties relative to solvability of certain systems of polynomial equations. 
For example, while finite fields and $\cc(\!(t)\!)$ are so-called $\mc{C}_1$-fields, in \cite{Ax} J.~Ax constructed a quasi-finite field which is not a $\mc{C}_1$-field.
The observation that strong linkage holds for algebraic function fields over finite fields (by \cite{Lenstra}) and over $\cc(\!(t)\!)$ (by \Cref{arithmeticcaseex})  motivates the following question.

\begin{qu}\label{Q:qf-sl}
Let $E$ be a quasi-finite field.
Is $\K_2^{(m)}F$ strongly linked for every algebraic function field $F/E$ and every positive integer $m$?
\end{qu}

We do not know the answer to this question even for $m=2$.
However, we can give a positive answer in the case where $E$ carries a discrete valuation.
This relies on the following characterisation of this situation.

\begin{thm}\label{T:qf-Z}
Let $E$ be a field and let $v$ be a $\zz$-valuation on $E$.
Then $E$ is quasi-finite if and only if  $v$ is henselian, $\kappa_v$ is algebraically closed and $\car(\kappa_v) =0$.
\end{thm}

Most of the implications contained in this statement follow immediately from classical results in valuation theory.
We include a complete elementary proof.
\medskip

\begin{proof}
Let $\pi\in \mg{E}$ be such that $v(\pi)=1$.

Assume that $\kappa_v$ is algebraically closed with $\car(\kappa_v)=0$ and that $v$ is henselian.
Then $\car(E)=0$, hence $E$ is perfect.
It follows by Hensel's Lemma that $E$ contains all roots of unity.
Since $v(\pi)=1$, for any positive integer $n$, the extension $E(\sqrt[n]{\pi})/E$ has  degree $n$. 
Let $L/E$ be a finite field extension and $d = [L:E]$. 
Since $v$ is henselian, $v$ has a unique extension $w$ to $L$.
Since $\kappa_v$ is algebraically closed, the residue field of $w$ is equal to $\kappa_v$, and hence by \cite[Theorem 3.3.5]{EP05} we have that
$ [w(\mg L): \zz]  = [L:E ] =d$.
We fix $\delta \in L$ such that $w(\delta) =\frac{1}{d}$. 
Then $\frac{\delta^d}{\pi}\in\mg{\mc{O}}_w$.
Since $\kappa_v$ is algebraically closed of characteristic zero, the assumption that $w$ is henselian yields that the polynomial $X^d-\frac{\delta^d}{\pi}$ has a root in $\mc{O}_w[X]$. Hence there exists 
$u\in\mc{O}_w$ with $\delta^d=\pi u^d$. We conclude that $L=E(\sqrt[d]{\pi})$.
Hence for any positive integer $d$ the unique field extension of $E$ of degree $d$ is given by $E(\sqrt[d]{\pi})$, which proves that $E$ is quasi-finite.

Assume now that $E$ is quasi-finite.
Since $E$ is perfect and carries a $\zz$-valuation, we have that $\car(E) =0$. 
For any positive integer $d$, since $X^d- \pi$ is irreducible in  $E[X]$, we obtain a field extension of degree $d$ of $E$ in which $v$ is totally ramified.
As $E$ has a unique extension of any degree, we obtain that $v$ is totally ramified in every finite field extension of $E$. 
Therefore $\kappa_v$ is algebraically closed and $v$ extends uniquely to any finite field extension of $E$. 
This shows that $v$ is henselian.

Suppose now that $\car(\kappa_v)=p$ for a prime number $p$.
As $\car(E)=0$, the polynomial $X^p-\pi$ is separable over $E$. 
Let $E'$ denote the root field of $X^p-\pi$ over $E$.
Since $E$ is quasi-finite and $v(\pi)=1$, it follows that $E'/E$ is a cyclic extension of degree $p$.
Hence $X^p-\pi$ splits over $E'$. 
By taking the quotient of two distinct roots of $X^p-\pi$ we obtain a primitive $p$th root of unity $\omega\in E'$.
As $[E(\omega):E]<p=[E':E]$, we find that $\omega\in E$.
Since every extension of degree $p^2$ of $E$ is cyclic, the group $\mg{E}/E^{\times p}$ is generated by the class of $\pi$.
As $v(\pi)=1$ and $v(1+\pi)=0$, it follows that $1+\pi\in E^{\times p}$. 
Hence 
$1+\pi=\xi^p$ for some $\xi\in E$.
Let $\vartheta=\xi-1$.
Since $\car(\kappa_v)=p$ it follows that $v(\vartheta)>0$. 
On the other hand, $\pi=(1+\vartheta)^p-1=\sum_{i=1}^p\binom{p}{i}\vartheta^i$.
Since $v(p)\geq 1$ and $v(\vartheta)\geq 1$ while $v(\pi)=1$,
we have a contradiction. 
This proves that $\car(\kappa_v)=0$.
\end{proof}

\begin{cor}\label{qqfsl}
Let $E$ be a quasi-finite field and $m$ a positive integer. 
If $E$ carries a $\zz$-valuation, then $\K_2^{(m)}F$ is strongly linked for any algebraic function field $F/E$.
\end{cor}

\begin{proof}
Let $v$ be a $\zz$-valuation on $E$.
By \Cref{T:qf-Z}, $\mc O_v$ is a henselian discrete valuation ring and $\kappa_v$ is algebraically closed with $\car(\kappa_v) =0$.
In particular $\car(\mc O_v) =0$. Hence $\mc O_v$ is excellent,
 by \Cref{excellentrem}.
Therefore $\K_2^{(m)}F$ is strongly linked, by \Cref{arithmeticcase}.
\end{proof}

We point out a type of examples of quasi-finite fields for which we do not know the answer to \Cref{Q:qf-sl}.
\begin{ex}\label{E:qf}
Let $k$ be a quasi-finite field of characteristic zero and let $E$ be the field of Puiseux series in $t$ over $k$, i.e.~$E=\bigcup_{r=1}^\infty k(\!(t^{{1}/r})\!)$.
This field $E$ carries a henselian valuation $v$ with value group $\qq$ and residue field $k$.
Since the value group of $v$ is divisible, $v$ is unramified in every finite extension of $E$.
On the other hand, since $v$ is henselian, for every positive integer $d$, any finite extension of $k$ lifts uniquely to an unramified extension of $E$ of the same degree.
Since the field $k$ has a unique field extension of degree $d$ for any positive integer $d$, we conclude that the same holds for $E$.
Hence $E$ is quasi-finite.
As $E$ does not carry any $\zz$-valuation, \Cref{T:qf-Z} does not help to answer \Cref{Q:qf-sl} in this case.
\end{ex}

\subsubsection*{Acknowledgments}
The authors express their gratitude to Jean-Louis Colliot-Th\'el\`ene, Arno Fehm, David Grimm, Gonzalo Manzano Flores, R.~Parimala, Suresh Venapally and Jan Van Geel  for various answers, discussions, suggestions, simplifications and other valuable input related to this article.
They further gratefully acknowledge the referee's comments and suggestions, which helped to streamline the presentation.

\bibliographystyle{abbrv}

\begin{thebibliography}{10}


\bibitem{AB}
M.~Auslander and A.~Brumer. 
{Brauer group of discrete valuation rings}.
\emph{Nederl.~Akad.~Wetensch.~Proc.~Ser.} {A}
 \textbf{71}~(1968), 286--296.
  
  
\bibitem{AG}
M.~Auslander and O.~Goldman. 
{The Brauer group of a commutative ring}.
\emph{Trans.~Amer.~Math.~Soc.} 
\textbf{97} (1960), 367--409.

\bibitem{Ax}
J.~Ax.
{Proof of some conjectures on cohomological dimension}.
\emph{Proc.~Amer.~Math.~Soc.} \textbf{16} (1965), 1214--1221. 

  
\bibitem{ChTi19}
A.~Chapman and J.-P.~Tignol.
{Linkage of Pfister forms over $\cc(x_1,\dots,x_n)$}.
\emph{Ann. $K$-Theory}, {\bf 4} (2019), 521--524.

\bibitem{CTOP}
J.-L. Colliot-Th\'el\`ene, M.~Ojanguren, and R.~Parimala. 
{Quadratic forms over fraction fields of $2$-dimensional henselian rings and Brauer groups of related schemes}.
\emph{Tata Inst. Fund. Res., Mumbai, Part II}, \textbf{16} (2002),
  185--217.
  

\bibitem{KC}
K.~\v{C}esnavi\v{c}ius. 
{Purity for the Brauer group}. 
\emph{Duke Math.~J.}, \textbf{168} (2019), 1461--1486. 

\bibitem{jong}
A.~J. de~Jong.
{The period-index problem for the Brauer group of an algebraic surface}. 
\emph{Duke Math.~J.}~\textbf{123} (2004),
71--94.  

\bibitem{EGA}
J.~Dieudonn{\'e} and A.~Grothendieck. 
{\'{E}l\'ements de g\'eom\'etrie alg\'ebrique: IV}.
\emph{Inst. Hautes~\'Etudes~Sci.~Publ.~Math.} \textbf{32}
  (1967), 5--361.
  
   
    
\bibitem{Dra}
P.~Draxl.
\emph{Skew fields}. 
London Math.~Soc. 
Lecture Note Series, vol.~{\bf 81}. 
Cambridge University Press, 1983.


\bibitem{Efrat}
I.~Efrat.
{A Hasse principle for function fields over PAC fields}. 
\emph{Israel J.~Math.}~\textbf{122} (2001), 43--60.

\bibitem{EP05}
A.~J.~Engler and A.~Prestel.
\newblock {\em  Valued fields}. 
\newblock Springer Monographs in Mathematics. Springer-Verlag, Berlin, 2005.

\bibitem{Ford96}
T.J. Ford. 
{Brauer group of a curve over a strictly local discrete valuation ring}.
\emph{Israel J.~Math.}, {\bf 96} (1996), 259--266.
  
  
  
\bibitem{FS89}
T.~J. Ford and D.~J. Saltman. 
{Division algebras over henselian surfaces}.
\emph{Ring theory 1989 (Ramat Gan and Jerusalem, 1988/1989)}, 320--336. {Israel Israel Math. Conf. Proc.}, \textbf{1}, Weizmann, Jerusalem, 1989. 

\bibitem{FJ}
M.D.~Fried and M.~Jarden. \emph{Field arithmetic. Third edition.} Revised by Jarden. Ergebnisse der Mathematik und ihrer Grenzgebiete. 3.~Folge, {\bf 11}. Springer-Verlag, Berlin, 2008.
 
\bibitem{GS}
P.~Gille and T.~Szamuely.
 \emph{Central simple algebras and Galois cohomology}.
  Cambridge University Press, 2006.
  
\bibitem{GroEGA}
A.~Grothendieck. 
{{\'El\'ements de g\'eom\'etrie alg\'ebrique: II. \'Etude
globale \'el\'ementaire de quelques classes de morphismes}}. 
\emph{Publications  Math\'ematiques de l'I.H.E.S} \textbf{8} (1961), 5--222.  
  
  

\bibitem{G2}
A.~Grothendieck.  
{Le groupe de Brauer, $iii$}, Exemples et compl\'ements.
In: 
 \emph{Dix expos\'es sur la cohomologie des schemas}, pp.~88-188.
{Adv. Stud. Pure Math.}, \textbf{3}, North-Holland, Amsterdam, 1968. 

    
\bibitem{Lenstra}
H.~W. Lenstra. 
{$\K_2$ of global fields consists of symbols}. 
\emph{Algebraic $K$-theory} (\emph{Proc.~Conf., Northwestern Univ., Evanston, Ill., 1976}), pp.~69--73. 
{Lecture Notes in Math.}~\textbf{551}, Springer, Berlin, 1976.
  
  
 \bibitem{Lip75}
 J.~Lipman. 
{Introduction to resolution of singularities}.
\emph{Algebraic geometry} (\emph{Proc. Sympos. Pure Math.}, \textbf{29}, \emph{Humboldt State Univ., Arcata, Calif., 1974}), pp. 187--230. Amer.~Math.~Soc.,~Providence, R.I., 1975. 

\bibitem{Lip78}
J.~Lipman. Desingularization of two-dimensional schemes. {\em Ann.~of Math.} {\bf 107} (1978): 151--207. 
 
\bibitem{Liu}
Q.~Liu. 
\emph{Algebraic geometry and arithmetic curves}.
Oxford Graduate Texts in Mathematics, Oxford University Press, 2012.

\bibitem{Ma-CRT}
H.~Matsumura. 
\emph{Commutative ring theory}. 
Cambridge studies in advanced mathematics, no.~8, Cambridge University Press, 1986.


\bibitem{Mil70}
J.~Milnor.
\newblock Algebraic $K$-theory and quadratic forms.
\newblock {\em Invent. Math.} {\bf 9} (1970): 318--344.

\bibitem{PS98}
R.~Parimala, V.~Suresh.
Isotropy of quadratic forms over function fields of $p$-adic curves.
\emph{Inst. Hautes \'Etudes Sci.~Publ.~Math.~No.} \textbf{88} (1998), 129--150. 

\bibitem{PS10}
R.~Parimala, V.~Suresh.
The $u$-invariant of the function fields of $p$-adic curves. 
\emph{Ann.~of Math.} \textbf{172} (2010), 1391--1405. 

\bibitem{Pfister}
A. Pfister.
\newblock {\em Quadratic forms with applications to algebraic geometry and
  topology}.
\newblock London Mathematical Society Lecture Notes series, vol. {\bf 217},
\newblock Cambridge University Press, Cambridge, 1995.

  \bibitem{Reiner}
I.~Reiner. 
\emph{Maximal orders}. 
London Mathematical Society Monographs New Series, 
Oxford University Press, 2003.


  
  \bibitem{Sal97}
D.J. Saltman. 
{Division algebras over $p$-adic curves}. 
\emph{J. Raman. Math. Soc.}~\textbf{12} (1997), 25--47.
  
  \bibitem{Saltman}
D.J.~Saltman. 
\emph{Lectures on division algberas}. 
CBMS Regional conference series in mathematics. no.~94 American Mathematical Society, 1999.

\bibitem{Sal07}
D.J.~Saltman. {Cyclic algebras over $p$-adic curves}, \emph{J. Algebra} \textbf{314}  (2007), 817--843.

  
\end{thebibliography}

\end{document}